\documentclass[10pt, a4paper, reqno]{amsart}

\usepackage{amscd, dsfont, amsmath, amstext, amsthm, amssymb, amsopn, amsxtra, amsfonts, setspace, fancyhdr, paralist}

\usepackage[mathcal]{euscript}
\usepackage[english]{babel}

\usepackage{tikz-cd}
\usepackage{boxedminipage}

\usepackage{pgfplots}
\pgfplotsset{mystyle/.style={color=red,no marks,line width=0.75pt}} 

\usepackage{todonotes}

\makeatletter
\renewcommand\@makefntext[1]{\hskip-0.7em\@makefnmark#1}
\makeatother

\newcommand{\hatt}{\scriptscriptstyle\wedge}

\newcommand{\ind}[1]{\operatorname{ind}_{#1}\nolimits}

\newcommand{\proj}[1]{\mathop{\operatorname{proj}}_{#1}\nolimits}

\newcommand{\T}{(T(t))_{t\geqslant0}}

\newcommand{\Cnull}{$\text{C}_0$}

\newcommand{\id}{\operatorname{id}\nolimits}

\makeatletter
\g@addto@macro\normalsize{%
  \setlength\abovedisplayskip{8pt}
  \setlength\belowdisplayskip{8pt}
  \setlength\abovedisplayshortskip{5pt}
  \setlength\belowdisplayshortskip{8pt}
}

\allowdisplaybreaks

\newtheoremstyle{normal}
{5pt}
{5pt}
{\normalfont}
{}
{\bfseries}
{}
{0.4em}
{\bfseries{\thmname{#1}\thmnumber{ #2}.\thmnote{ \hspace{0.5em}(#3)\newline}}}

\newtheoremstyle{kursiv}
{5pt}
{5pt}
{\itshape}
{}
{\bfseries}
{}
{0.4em}
{\bfseries{\thmname{#1}\thmnumber{ #2}.\thmnote{ \hspace{0.5em}(#3)\newline}}}

\theoremstyle{kursiv}

\newtheorem{thm}{Theorem}

\newtheorem{cor}[thm]{Corollary}
\newtheorem{lem}[thm]{Lemma}

\theoremstyle{normal}
\newtheorem{dfn}[thm]{Definition}

\newtheorem{rem}[thm]{Remark}

\usepackage{soul}
\newcommand{\red}[1]{{\leavevmode\color{black}#1}}

\renewcommand{\epsilon}{\varepsilon}
\renewcommand{\phi}{\varphi}

\DeclareFontEncoding{FML}{}{}
\DeclareFontSubstitution{FML}{futm}{m}{it}

\DeclareSymbolFont{gletters}{FML}{futm}{m}{it}

\DeclareMathSymbol{\PHI}{\mathord}{gletters}   {30}
\renewcommand{\Phi}{\PHI}

\DeclareMathSymbol{\PSI}{\mathord}{gletters}  {32}
\renewcommand{\Psi}{\PSI}

\begin{document}
$ $
\vspace{-35pt}

\title{Pivot duality of universal interpolation\\and extrapolation spaces}

\author{Christian Bargetz\hspace{0.5pt}\MakeLowercase{$^{\text{1}}$} and Sven-Ake Wegner\hspace{0.5pt}\MakeLowercase{$^{\text{2}}$}}

\renewcommand{\thefootnote}{}
\hspace{-1000pt}\footnote{\hspace{5.5pt}2010 \emph{Mathematics Subject Classification}: Primary 47A06, Secondary 47D03, 46A13.}
\hspace{-1000pt}\footnote{\hspace{5.5pt}\emph{Key words and phrases}: Dual with respect to a pivot space, extrapolation space, interpolation space.\vspace{1.6pt}}

\hspace{-1000pt}\footnote{\hspace{0pt}$^{1}$\,University of Innsbruck, Department of Mathematics, Technikerstra{\ss}e~13, 6020 Innsbruck, Austria (current address) and The Technion---Israel Institute of Technology, Department of Mathematics, Haifa 320000, Israel,\\ E-Mail: christian.bargetz@uibk.ac.at.\vspace{1.6pt}}

\hspace{-1000pt}\footnote{\hspace{0pt}$^{2}$\,Corresponding author: University of Wuppertal, School of Mathematics and Natural Sciences, Gau{\ss}stra{\ss}e 20, 42119 Wuppertal, Germany (current address) and Nazarbayev University, School of Science and Technology, 53 Kabanbay Batyr Ave, 010000 Astana, Kazakhstan, E-Mail: wegner@math.uni-wuppertal.de, Phone: +49\hspace{1.2pt}(0)\hspace{1.2pt}202\hspace{1.2pt}/\hspace{1.2pt}439\hspace{1.2pt}-\hspace{1.2pt}2590.}

\begin{abstract} It is a widely used method, for instance in perturbation theory, to associate with a given \Cnull-semigroup its so-called interpolation and extrapolation spaces. In the model case of the shift semigroup acting on $L^2(\mathbb{R})$, the resulting chain of spaces recovers the classical Sobolev scale. In 2014, the second named author defined the universal interpolation space as the projective limit of the interpolation spaces and the universal extrapolation space as the completion of the inductive limit of the extrapolation spaces, provided that the latter is Hausdorff.

\smallskip

\noindent{}In this note we use the notion of the dual with respect to a pivot space in order to show that the aforementioned inductive limit is Hausdorff, already complete, and can be represented as the dual of the projective limit whenever a power of the generator of the initial semigroup is a self-adjoint operator. In the case of the classical Sobolev scale we show that the duality holds, and that the two universal spaces were already studied by Laurent Schwartz in the 1950s.

\smallskip
  
\noindent{}Our results and examples complement the approach of Haase, who in 2006 gave a different definition of universal extrapolation spaces in the context of functional calculi. Haase avoids the inductive limit topology precisely for the reason that it a priori cannot be guaranteed that the latter is always Hausdorff. We show that this is indeed the case provided that we start with a semigroup defined on a reflexive Banach space.
\end{abstract}

\maketitle
\vspace{-15pt}
\section{The classical Sobolev scale}\label{SEC-1}
\smallskip

We start by considering the following generic example. 
Let $\T$ denote the left shift semigroup on the Hilbert space $L^2(\mathbb{R})$ generated by the first 
derivative $\frac{\rm d}{{\rm d}x}$ defined on the domain $D(\frac{\rm d}{{\rm d}x})=\{f\in L^2(\mathbb{R})\:;\:\frac{\rm d}{{\rm d}x}f\in L^2(\mathbb{R})\}$. Writing down the 
abstract interpolation and extrapolation spaces \cite[Chapter II.5]{EN} gives the classical scale of 
Sobolev spaces
\[
\cdots \longrightarrow \mathcal{H}^{3}(\mathbb{R}) \stackrel{i_3^2}{\longrightarrow} \mathcal{H}^{2}(\mathbb{R}) \stackrel{i_2^1}{\longrightarrow} \mathcal{H}^1(\mathbb{R}) \stackrel{i_1^0}{\longrightarrow} L^2(\mathbb{R}) \stackrel{i_0^{-1}}{\longrightarrow} \mathcal{H}^{-1}(\mathbb{R}) \stackrel{i_{-1}^{-2}}{\longrightarrow} \mathcal{H}^{-2}(\mathbb{R})\stackrel{i_{-2}^{-3}}{\longrightarrow} \mathcal{H}^{-3}(\mathbb{R}) \rightarrow \cdots
\]
where the maps are all continuous. Taking the projective limit of this chain of spaces, i.e., endowing 
the intersection $\cap_{n\in\mathbb{N}}\mathcal{H}^n(\mathbb{R})$ with the coarsest linear topology which 
makes the inclusions $\cap_{n\in\mathbb{N}}\mathcal{H}^n(\mathbb{R})\rightarrow\mathcal{H}^k(\mathbb{R})$ 
for all $k\in\mathbb{N}$ continuous, yields the classical function space
\[
\mathcal{D}_{L^2}(\mathbb{R})=\proj{n\in\mathbb{N}}\mathcal{H}^n(\mathbb{R})
\]
studied by Schwartz \cite[\S\,8, p.~199]{Schwartz}. Taking the inductive limit, i.e., endowing the union 
$\cup_{n\in\mathbb{N}}\mathcal{H}^{-n}(\mathbb{R})$ with the finest linear topology which makes the inclusions 
$\mathcal{H}^k(\mathbb{R})\rightarrow\cup_{n\in\mathbb{N}}\mathcal{H}^{-n}(\mathbb{R})$ for all $k\in\mathbb{N}$ 
continuous, yields a subspace of the space of distributions which turns out to be isomorphic to the strong 
dual of $\mathcal{D}_{L^2}(\mathbb{R})$, i.e.,
\[
\mathcal{D}'_{L^2}(\mathbb{R}) \cong \ind{n\in\mathbb{N}}\mathcal{H}^{-n}(\mathbb{R})
\]
in a natural way. Also this space was investigated by Schwartz \cite[\S\,8, p.~200]{Schwartz}. Indeed, we 
have the following commutative diagram
\begin{equation}\label{EQ-1}
\begin{tikzcd}
 L^2(\mathbb{R}) \arrow{r}{i_0^{-1}} \arrow{d}{\id_{L^2(\mathbb{R})}} & \mathcal{H}^{-1}(\mathbb{R}) \arrow{r}{i_{-1}^{-2}}\arrow{d}{\Phi_1} & \mathcal{H}^{-2}(\mathbb{R})\arrow{r}{i_{-2}^{-3}} \arrow{d}{\Phi_2} & \mathcal{H}^{-3}(\mathbb{R})  \arrow{r}{} \arrow{d}{\Phi_3} & \cdots  \\              
L^2(\mathbb{R}) \arrow[swap]{r}{(i_1^0)'} & \mathcal{H}^{1}(\mathbb{R})' \arrow[swap]{r}{(i_2^1)'} & \mathcal{H}^{2}(\mathbb{R})' \arrow[swap]{r}{(i_3^2)'} & \mathcal{H}^{3}(\mathbb{R})' \arrow{r}{} & \cdots 
\end{tikzcd}
\end{equation}
where the maps $\Phi_n$ for $n\in\mathbb{N}$ are isomorphisms. Our first aim is to see that the corresponding inductive limits are isomorphic. We emphasize that this is not trivial just by having \textquotedblleft{}step-wise\textquotedblright{} isomorphisms. Indeed, we have for instance $\mathcal{H}^{-n}(\mathbb{R})\cong L^2(\mathbb{R})$ for each $n\in\mathbb{N}$ but $L^2(\mathbb{R})\not\cong\cup_{n\in\mathbb{N}}\mathcal{H}^{-n}(\mathbb{R})$, which shows that we have to be extremely careful when we \textquotedblleft{}identify\textquotedblright{} isomorphic spaces.

\smallskip

\color{black}

The suitable notion to address our first aim is that of \textit{equivalent inductive sequences}. Each row in the diagram \eqref{EQ-1} is a so-called inductive sequence, i.e., a sequence $(X_n,i_n^{n+1})_{n\in\mathbb{N}}$ of Banach spaces $X_n$ and linear and continuous maps $i_n^{n+1}\colon X_n\rightarrow X_{n+1}$ for $n\in\mathbb{N}$. Two such inductive sequences $(X_n,i_n^{n+1})_{n\in\mathbb{N}}$ and $(Y_n,j_n^{n+1})_{n\in\mathbb{N}}$ are said to be equivalent, if there are increasing sequences $(k(n))_{n\in\mathbb{N}}$ and $(\ell(n))_{n\in\mathbb{N}}$ of natural numbers with $n\leqslant \ell(n)\leqslant k(n)\leqslant \ell(n+1)$ and  linear and continuous maps $\alpha_n\colon Y_{\ell(n)}\rightarrow X_{k(n)}$, $\beta_n\colon X_{k(n)}\rightarrow Y_{\ell(n+1)}$  such that 
\begin{center}
\includegraphics[width=290pt]{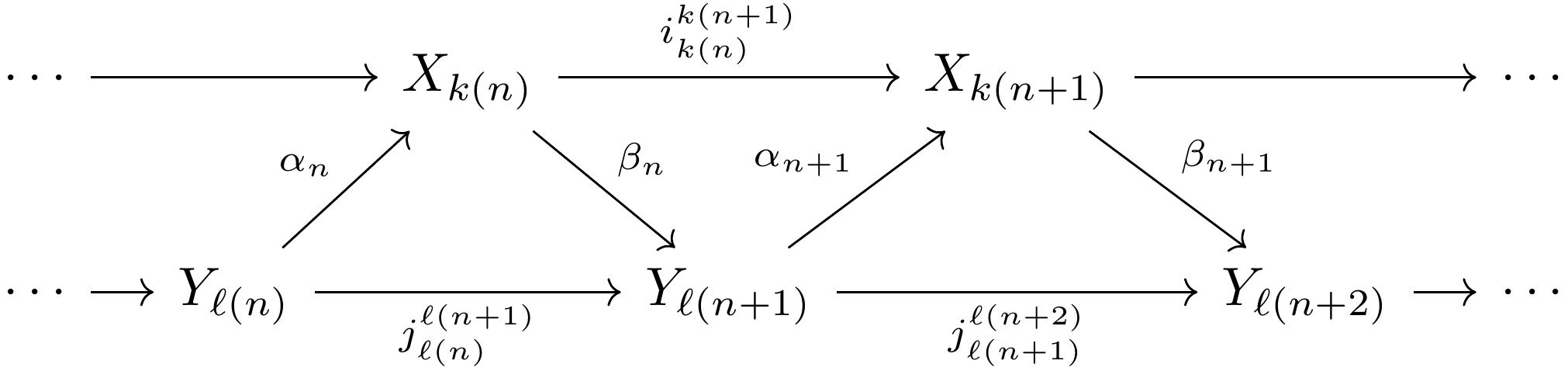}
\end{center}
commutes. As a matter of fact, two equivalent inductive sequences have isomorphic inductive limits.

\smallskip

We now see that the two inductive sequences in \eqref{EQ-1} are equivalent and we thus get that 
\[
\ind{n\in\mathbb{N}}\mathcal{H}^{-n}(\mathbb{R})\cong \ind{n\in\mathbb{N}}\mathcal{H}^{n}(\mathbb{R})'
\]
holds. Now we would like to conclude that the dual of a projective limit\,(=intersection) is equal to the inductive limit\,(=union) of the duals of the spaces in the sequence. This is indeed true but requires an open mapping theorem for topological vector spaces. Note that the important conclusion of the following argument is that we finally get the isomorphism \eqref{EQ-2} below. \color{black} We observe that $\mathcal{D}_{L^2}(\mathbb{R})$ is a reflexive Fr\'{e}chet space by Jarchow \cite[Proposition~11.5.5 and Corollary 11.4.3]{J} and thus distinguished by \cite[Remark after 13.4.5 on p.~280]{J}. 
Then by \cite[Corollary~13.4.4]{J} it follows that $\mathcal{D}'_{L^2}(\mathbb{R})$ is ultrabornological. As $\mathcal{H}^{-n}(\mathbb{R})\subseteq\mathcal{H}^{-(n+1)}(\mathbb{R})$ is dense for all $n\in\mathbb{N}$, we have that the corresponding sequence is reduced and \cite[Proposition 8.8.7]{J} implies that there is a natural isomorphism
\begin{equation}\label{EQ-2}
  \ind{n\in\mathbb{N}}\mathcal{H}^{n}(\mathbb{R})'\stackrel{\phi}{\longrightarrow} (\proj{n\in\mathbb{N}}\mathcal{H}^n(\mathbb{R}))'_b = \mathcal{D}'_{L^2}(\mathbb{R}) 
\end{equation}
of linear spaces. As $\phi|_{\mathcal{H}^n(\mathbb{R})'}\colon \mathcal{H}^n(\mathbb{R})' \rightarrow 
(\proj{n\in\mathbb{N}}\mathcal{H}^n(\mathbb{R}))'_b$ is continuous, it follows that $\phi$ is continuous and hence an isomorphism by the open mapping theorem \cite[Theorem 5.5.2]{J} as its domain is---as an LB-space---in particular webbed \cite[Corollary 5.3.3]{J} and its codomain is ultrabornological.

\medskip

We point out again, that the crucial fact for the above arguments is not just the existence of isomorphisms $\Phi_n$. In the current example this indeed might appear to be trivial by a well-known characterization of the Sobolev spaces, e.g., that $\mathcal{H}^{-n}(\mathbb{R})$ can be identified with the dual of $\mathcal{H}^{n}(\mathbb{R})$. However, exactly these identifications have to be compatible with the maps in the inductive sequences in order to induce an equivalence of the latter. Below, in Definition \ref{DFN-PD}, we employ a generalization of the notion of the \textquotedblleft{}dual with respect to a pivot space\textquotedblright{}, see Tucsnak, Weiss \cite[Chapter 2.10]{TW}, in order to construct families of isomorphisms that are compatible with the linking maps of the corresponding inductive sequences. The gap that we left in the example above will be closed by Corollary \ref{COR-2}.

\medskip

In the notation established by the second-named author in \cite{W}, the space $\proj{n\in\mathbb{N}}\mathcal{H}^n(\mathbb{R})$ is the \textit{universal interpolation space} associated with the shift semigroup on $L^2(\mathbb{R})$. According to the latter article, the \textit{universal extrapolation space} is defined to be the completion of $\ind{n\in\mathbb{N}}\mathcal{H}^{-n}(\mathbb{R})$. By the representation above and as the strong dual of a Fr\'{e}chet space is complete, see e.g. Meise, Vogt \cite[p.~296]{MV}, we see however that the inductive limit is already complete and thus forming the completion is dispensable.

\medskip

Haase \cite[p.~143ff]{Haase-Buch} and \cite[p.~221ff]{Haase} proposed a different notion of \textit{universal extrapolation space}: In the situation of our example, his extrapolation space is the union $U=\cup_{n\in\mathbb{N}}\mathcal{H}^{-n}(\mathbb{R})$ in which a net $(x_{\alpha})_{\alpha\in\mathcal{A}}\subseteq U$ converges by definition to $x\in U$ if
$$
\exists\:n\in\mathbb{N},\,\alpha_0\in\mathcal{A}\;\forall\:\alpha\geqslant\alpha_0\colon x,\,x_{\alpha}\in \mathcal{H}^{-n}(\mathbb{R}) \,\text{ and }\, x_\alpha\rightarrow x \,\text{ in }\,\mathcal{H}^{-n}(\mathbb{R})
$$
holds. Haase \cite[Remark 8.1.2]{Haase} motivates his definition by the fact that the inductive topology even in the case of a countable sequence of Banach spaces needs not to be Hausdorff.

\medskip

In \cite{W} several sequence space cases were studied in which the inductive limit of the extrapolation spaces turned out to be Hausdorff and to be complete. In this note we go one better and provide a general theorem that establishes that the inductive limit is Hausdorff and even complete whenever our initial Banach space is reflexive. Further, we identify conditions under which the duality of universal inter- and extrapolation space holds. Above, we gave already one example in which the corresponding limit space is even well-known since a long time ago.

\medskip
\section{Duals with respect to a pivot space}\label{SEC-2}

\smallskip

\red{For the whole section let $X$ be a reflexive Banach space with norm $\|\cdot\|_{X}$ and 
\[
A\colon D(A) \rightarrow X
\]
be a closed and densely defined operator on $X$ such that $A\colon D(A)\rightarrow X$ is invertible with $A^{-1}\in L(X)$. In particular we may think of $A$ as the generator of a \Cnull-semigroup on $X$. Note that in the latter case one can assume w.l.o.g.\,that $A^{-1}$ exists and belongs to $L(X)$, see \cite[p.~124]{EN}.

\medskip

Following~\cite{EN,W}, we introduce the interpolation and extrapolation spaces with respect to the operator~$A$. 
The \emph{$n$-th interpolation space} is defined as 
\[
X_n := (D(A^{n}), \|\cdot\|_{n}), \;\; \text{where} \;\; \|x\|_{n} := \|x\|_{X} + \|A^nx\|_{X},\;\;\text{for}\;\; x\in X_n.
\]
Note that the norm $\|\cdot\|_{n}$ is the graph norm with respect to~$A^{n}$. For the definition of the extrapolation
spaces, we define the norm
\[
\|x\|_{-n}=\|A^{-n}x\|_X, \;\; \text{where}\;\; A^{-n} := \left(A^{-1}\right)^{n},
\]
on $X$. We define the \emph{$n$-the extrapolation space}
\[
X_{-n}:=(X,\|\cdot\|_{-n})^{\hatt}
\]
as the completion of $X$ with respect to the norm $\|\cdot\|_{-n}$.

\medskip

As in the motivating example in the first section, we want to find a representation of the inductive limit of the spaces $X_{-n}$ as the dual space of a projective limit of a suitable sequence of Banach spaces. Before we start this construction, we first summarize some facts which follow easily from~\cite[Chapter 1]{N} and~\cite[Appendix B]{EN}. Given a Banach space $X$ with dual space $X'$, we denote by $\langle\cdot,\cdot\rangle$ the duality mapping between $X$ and $X'$, i.e., $\langle x, \phi\rangle := \phi(x)$, for $x\in X$ and $\phi\in X'$. We put
$$
D(A')=\bigl\{x'\in X'\:;\: \exists\:y'\in X'\;\forall\:x\in D(A)\colon \langle{}Ax,x'\rangle{}=\langle{}x,y'\rangle{}\bigr\}
$$
and denote by $A'\colon D(A')\rightarrow X'$, $A'x'=y'$ the dual operator of $A$. When we consider its powers we 
abbreviate $A'^{\,n}:=(A')^n$.

\begin{lem}\label{LEM-0} Let $X$ be a reflexive Banach space.
\begin{compactitem}\vspace{3pt}
\item[(i)] The operator $A'\colon D(A')\rightarrow X'$ is closed and densely defined.\vspace{3pt}
\item[(ii)] For $n\geqslant1$ the operator $A^n$ is invertible and we have \red{$(A^n)^{-1}=A^{-n}\in L(X)$}.\vspace{3pt}
\item[(iii)] For $n\geqslant 1$ we have $(A^{-n})'=(A')^{-n}$.\hfill\qed
\end{compactitem}
\end{lem}

Observe that it is well-known, that~(i) fails without reflexivity and that the proof of~(iii) relies on~(i). Remember that we only consider reflexive spaces, although some single arguments might hold true without this assumption, as, e.g., part~(ii) in the lemma above.

\medskip

Using that $X$ is reflexive, the main idea of the following construction is to find suitable subspaces of $X'$ which, up to isomorphisms, may act as preduals of the extrapolation spaces $X_{-n}$.

\begin{dfn}\label{DFN-PD} Let $X$ be reflexive and $n\geqslant1$. For $x\in X$ we put
\[
\|x\|_{\star,n}=\sup\big\{|\langle{}x,\psi\rangle{}|\:;\:\|\psi\|_{D(A'^{\,n})}\leqslant1\bigr\}
\]
which defines \red{a family} of new norms on $X$. Using these norms, we denote by
\begin{equation*}
X_n^d:=(D(A'^{\,n}),\|\cdot\|_{D(A'^{n})})
\end{equation*}
the $n$-th interpolation space with respect to the dual operator $A'$. Generalizing the Hilbert space notation from~\cite[p.~60]{TW}, we put
\[
(X_n^d)^{\star}:=(X,\|\cdot\|_{\star,n})^{\hatt}
\]
and call this space the \textquotedblleft{}dual of $X_n^d$ with respect to the pivot space $X$\textquotedblright{}. 
\end{dfn}}

The spaces $X_n^d$ defined above will,  up to isomorphisms, act as the aforementioned preduals of the extrapolation spaces~$X_{-n}$. In order to see this relation, we need the spaces $(X_n^d)^\star$. The following lemma is used to establish isomorphisms between $X_{-n}$ and $(X_n^d)^\star$. In the case where $X$ is Hilbert and $n=1$, it is proved in \cite[Proposition~2.10.2]{TW}.

\medskip

\begin{lem}\label{LEM-1} Let $X$ be reflexive and $n\geqslant1$. Then, the norms $\|\cdot\|_{\star,n}$ and $\|\cdot\|_{-n}$ are equivalent.
\end{lem}

\begin{proof} 
  Let $x\in X$ be given. From Lemma \ref{LEM-0} we get $\langle{}A^{-n}x,\varphi\rangle{}=\langle{}x, (A^{-n})'\varphi\rangle{}=\langle{}x,(A')^{-n}\varphi\rangle{}$ and thus we obtain
  \[
    \|x\|_{-n}=\|A^{-n}x\|_X = \sup\bigl\{|\langle{}A^{-n}x,\varphi\rangle{}|\:;\:\|\varphi\|_{X'}\leqslant1\bigr\} =\sup\bigr\{|\langle{}x,(A')^{-n}\varphi\rangle{}|\:;\:\|\varphi\|_{X'}\leqslant1\bigr\}.
  \]
  Since $A'^{\,n}\colon D(A'^{\,n})\rightarrow X'$ is bijective with inverse $(A')^{-n}$ we can substitute $\psi=(A')^{-n}\varphi$ and get from the above
  \[
    \|x\|_{-n}  = \sup\bigl\{|\langle{}x,\psi\rangle{}|\:;\:\|A'^{\,n}\psi\|_{X'}\leqslant1\bigr\}.
  \]
  If we endow $D(A'^{\,n})$ with the homogeneous norm $\|\cdot\|_{A',n}$, i.e., $\|x\|_{A',n}=\|A'^nx\|_X$, then the map $A'^{\,n}\colon D(A'^{\,n})\rightarrow X'$ is an isometric isomorphism. It follows
  \[
    \|x\|_{-n} =\sup\bigl\{|\langle{}x,\psi\rangle{}|\:;\:\|A'^{\,n}\psi\|_{X'}\leqslant1\bigr\} =\sup\bigl\{|\langle{}x,\psi\rangle{}|\:;\:\|\psi\|_{A',n}\leqslant1\bigr\}
  \]
  which shows that $\|\cdot\|_{-n}$ and $\|\cdot\|_{\star,n}$ are equivalent, as the homogeneous norm is equivalent to 
  the graph norm.
\end{proof}

\medskip

Since $A'^{\,n}\colon X_n^d\rightarrow X'$ is an isomorphism and $X$ is reflexive, $X_n^d$ is again reflexive. By Lemma~\ref{LEM-1} we can define 
\[
  \Phi_n\colon X_{-n}\rightarrow (X_n^d)^{\star}
\]
to be the unique extension of $\id_X$. Moreover, we define 
\[
  \Psi_n\colon (X_n^d)^{\star}\rightarrow (X_n^d)'
\]
by
\[
  [\Psi_n(z)](\varphi)=\lim_{k\rightarrow\infty}\langle{}z_k,\varphi\rangle_{X,X'} \;\; \text{ for }\;z\in (X_n^d)^{\star},\,\varphi\in X_n^d
\] 
where $(z_k)_{k\in\mathbb{N}}\subseteq X$ with $z_k\rightarrow z$ in $(X_n^d)^{\star}$. Finally, we denote by $i_n\colon X\rightarrow (X_n^d)^{\star}$ the inclusion and by $(j_n)'\colon X \rightarrow (X_n^d)'$ the dual of the inclusion $j_n\colon X_n^d\rightarrow X'$.

\medskip

\begin{lem}\label{PROP-2-Banach} 
  Let $X$ be reflexive and $n\geqslant1$. Then, $\Psi_n$ is an isomorphism and $\Psi_n\circ{} i_n=(j_n)'$ holds.
\end{lem}

\begin{proof} It is straightforward to see that $(\Psi_n(z))(\varphi)$ 
  is well-defined, i.e., that the limit exists and is independent of the choice of the sequence. The linearity 
  of $\Psi_n(z)$ and $\Psi_n$ follows from the definition of $\Psi_n$ and the continuity of the vector space 
  operations. From $\|z\|_{\star,n}=\sup\{|\langle z,\phi\rangle| \:;\: \|\phi\|_{D(A'^{\,n})}\leqslant 1\}$ we may deduce 
  $|\langle z,\varphi\rangle|\leqslant  \|\varphi\|_{D(A'^{\,n})} \|z\|_{\star,n}$ for $z\in X$ and $\varphi\in X_n^d$ and 
  by density for $z\in (X_n^d)^{\star}$. Hence
  \[
    \Psi_n(z) \in (X_n^d)' \; \text{ and } \; |(\Psi_n(z))(\varphi)| \leqslant \|\varphi\|_{X_n^d} \|z\|_{\star,n}.
  \]
  Therefore $\|\Psi_n(z)\|_{(X_n^d)'} \leqslant \|z\|_{\star,n}$ and $\Psi_n\colon (X_n^d)^{\star} \to (X_n^d)'$ is continuous.
  
  \medskip
  
  For $z\in X$ we have $(\Psi_n(z))(\varphi)=\langle z, \varphi\rangle$ since the constant sequence $(z)_{k\in\mathbb{N}}$ converges to $z$.
  Therefore we get for $z\in X$
  \[
    \|\Psi_n(z)\|_{(X_n^d)'} = \sup\{|\langle \Psi_n(z),\varphi\rangle| \:;\: \|\varphi\|_{X_n^d}\leqslant 1\} 
    = \sup\{|\langle z, \varphi \rangle| \:;\: \|\varphi\|_{X_n^d}\leqslant 1\} = \|z\|_{\star,n}.
  \]
  Since $X\subseteq (X_n^d)^\star$ is dense and $\Psi_n$ is continuous, we may conclude that $\|\Psi_n(z)\|_{(X_n^d)'}=\|z\|_{\star,n}$ holds for all $z\in(X_n^d)^\star$. In other words $\Psi_n$ is an isometric embedding and hence has a closed range. In order to finish the proof, we show that the range is also dense. Assume for a contradiction that $\Psi_n((X_n^d)^\star)\subseteq (X_n^d)'$ is not a dense subset. Then the Hahn-Banach Theorem implies the existence of a functional $0\neq\varphi\in (X_n^d)'' = X_n^d$ where $\langle \Psi_n(z),\varphi\rangle = 0$ for all $z\in (X_n^d)^\star$. In particular, we get $\langle z,\varphi \rangle = 0$ for all $z\in X$ which is a contradiction since $X_n^d\subseteq X'$ and therefore $0\neq\varphi\in X'$ with $\langle z,\varphi\rangle = 0$ for all $z\in X$. Finally, the equality
  \[
    (\Psi_n(i_n(z)))(\varphi) = \langle \Psi_n(z),\varphi\rangle = \langle z,\varphi \rangle = ((j_n)'(z))(\varphi)
  \]
  shows $\Psi_n\circ i_n = (j_n)'$.
\end{proof}

\medskip

We mention, that the above proof was inspired by~\cite[Proposition~2.9.2]{TW}, where a similar statement is shown when $X$ is a Hilbert space and $n=1$.

\medskip

\color{black}
\section{Universal inter- and extrapolation spaces}\label{SEC-3}

\smallskip

In this section, we now form inductive and projective limits of the \red{sequences} considered before. \red{The main result of this section is the following theorem which provides a representation of the inductive limit of the extrapolation spaces for a densely defined, closed operator having a bounded inverse. In the situation of~\cite{W}, where $A$ is the generator of a \Cnull-semigroup, this inductive limit was called the universal extrapolation space.
Note that in the latter case one can assume w.l.o.g.\,that $A^{-1}$ exists and belongs to $L(X)$, see \cite[p.~124]{EN}.}

\begin{thm}\label{THM} 
  Let $X$ be a reflexive Banach space and $A\colon D(A)\rightarrow X$ be a \red{densely defined, closed operator} such that 
  $A^{-1}$ exists and belongs to $L(X)$. Then we have
  \begin{equation}
    (\proj{n\in\mathbb{N}}X_n^d)'_b\cong \ind{n\in\mathbb{N}} X_{-n}.
  \end{equation}
  In particular, the inductive limit $\ind{n\in\mathbb{N}} X_{-n}$ is complete.
\end{thm}

In order to prove this theorem, we first have to specify the linking maps between the spaces defined in Section~\ref{SEC-2} in more detail. We start with the linking maps between the extrapolation spaces, i.e. the maps
\[
  i_{n}^{n+1}\colon X_{-n}\to X_{-(n+1)}.
\]
In~\cite[Remark after II.5.6]{EN}\red{, where $A$ is always assumed to be a generator of a \Cnull-semigroup,} it is mentioned, that
\begin{equation}\label{EN}
  X_{-(n+m)}=(X_{-(n)})_{-m}
\end{equation}
holds for all $n$, $m\geqslant1$, where $X_{-n}$ for $n\geqslant1$ is the $n$-th extrapolation space as defined above. Indeed, the above equality holds in the sense of isometric isomorphisms that constitute an equivalence of inductive sequences. We first observe that, given a closed and densely defined operator~$A$ with $0\in\rho(A)$, then $A_{-n}$ has the same properties. Therefore we can define 
\[
  Y_{-n}:=(\cdots(X)_{-1}\cdots)_{-1}
\]
by forming $n$-times the first extrapolation space. We denote by $j_{n}^{n+1}\colon Y_{-n}\rightarrow Y_{-(n+1)}$ the inclusion maps. 
As $X\subseteq Y_{-n}$ and $X\subseteq X_{-n}$ are dense and $\|x\|_{-n}=\|x\|_{Y_{-n}}$ holds for $x\in X$, the identity $\id_X$ extends 
to an isometric isomorphism $\Theta_n\colon Y_{-n}\rightarrow X_{-n}$. We define the maps $i_{n}^{n+1}\colon X_{-n}\rightarrow X_{-(n+1)}$ 
via $i_{n}^{n+1}:=\Theta_{n+1}\circ j_{n}^{n+1}\circ \Theta_n^{-1}$, i.e., the diagram
\begin{center}
  \begin{tikzcd}
    X \arrow{r} \arrow{d}{\id_X} &\cdots \arrow{r} & Y_{-n} \arrow{r}{j_n^{n+1}}\arrow{d}{\Theta_n} & Y_{-(n+1)} \arrow{r}\arrow{d}{\Theta_{n+1}}  &\cdots\\
    X \arrow{r} &\cdots \arrow{r} & X_{-n} \arrow[swap]{r}{i_n^{n+1}}& X_{-(n+1)} \arrow{r} &\cdots
  \end{tikzcd}
\end{center}
is commutative. Note that the mappings $\Theta_n$ are only used to define the linking maps $i_{n}^{n+1}$ and will not be used the the sequel.

\medskip

Since the space $(X_n^d)^{\star}$ is defined as the completion of $X$ with respect to the norm $\|\cdot\|_{\star,n}$, $X$ is dense in $(X_n^d)^{\star}$ for all $n\geqslant 1$. We can therefore define the mapping
\[
  (i_n^{n+1})^{\star}\colon (X_n^{d})^{\star}\to (X_{n+1}^{d})^{\star}
\]
as the extension of the identity on $X$. In addition, we define the mapping 
\[
  (i_n)^{\star}\colon X\to (X_n^d)^{\star}
\]
to be the canonical embedding of $X$ into $(X_n^d)^{\star}$. In other words, the above construction means that $(i_n^{n+1})^\star$ is the unique continuous linear mapping satisfying $(i_{n+1})^{\star}=(i_n^{n+1})^\star\circ (i_n)^\star$. Since $X_n^d$ is defined as $D(A'^{\,n})$ equipped with the graph norm, the mappings 
\[
  j_n := \id_X\!|_{D(A'^{\,n})}\colon X_n^d\to X', \; \text{ and } \;
  j_{n+1}^{n}:=\id_{X_{n}^d}\!|_{D(A'^{n+1})}\colon X_{n+1}^{d}\to X_{n}^{d}
\]
are continuous inclusions. By \cite[Proposition~6.2]{AMK1992} the operators also have dense images. Therefore, we obtain by duality the maps
\[
  (j_n)' \colon X\to (X_{n}^{d})' \; \text{ and } \; (j_{n+1}^{n})'\colon (X_{n}^{d})'\to (X_{n+1}^{d})'
\]
which satisfy $(j_{n+1})' = (j_{n+1}^{n})' \circ (j_n)'$.

\medskip
\color{black}
\begin{proof}[Proof of Theorem~\ref{THM}]
  We divide the proof into three steps.
  \medskip

  As a first step we show that the sequence $(\Phi_n)_{n\in\mathbb{N}}$ defined in Section~\ref{SEC-2} is an equivalence between the inductive sequences $(X_{-n},i_n^{n+1})_{n\in\mathbb{N}}$ and $((X_{n}^d)^{\star},(i_n^{n+1})^\star)_{n\in\mathbb{N}}$.
  We have to show that the diagram
  \begin{equation}\label{PROP-3}
    \begin{tikzcd}
      X \arrow{r} \arrow{d}{\id_X} &\cdots \arrow{r} & X_{-n} \arrow{r}{i_n^{n+1}}\arrow{d}{\Phi_n} & X_{-(n+1)} \arrow{r}\arrow{d}{\Phi_{n+1}}  &\cdots\\
      X \arrow{r} &\cdots \arrow{r} & (X_n^d)^{\star} \arrow[swap]{r}{(i_n^{n+1})^{\star}} & (X_{n+1}^d)^{\star} \arrow{r} &\cdots
    \end{tikzcd},
  \end{equation}
  is commutative. Note that by definition both $\Phi_n\circ (i_n^{n+1})^{\star}$ and $i_n^{n+1}\circ\Phi_{n+1}$, restricted to $X$, coincide with the identity on $X$, which implies the commutativity since $X$ is dense both in $X_{-n}$ and $(X_{n+1}^d)^\star$.

  \medskip

  As a second step, we show that the sequence of maps $(\Psi_n)_{n\in\mathbb{N}}$ is an equivalence between the inductive sequences $((X_{n}^d)^{\star},(i_n^{n+1})^{\star})_{n\in\mathbb{N}}$ and $((X_{n}^d)',(j_{n+1}^{n})')_{n\in\mathbb{N}}$. Here, we have a commutative diagram
  \begin{equation}\label{PROP-4-Banach}
    \begin{tikzcd}
      X \arrow{r}\arrow{d}{\operatorname{id}_X} &\cdots \arrow{r} & (X_n^d)^{\star} \arrow{r}{(i_n^{n+1})^{\star}} \arrow{d}{\Psi_n} 
      & (X_{n+1}^d)^{\star} \arrow{r}\arrow{d}{\Psi_{n+1}} &\cdots\\
      X \arrow{r} &\cdots \arrow{r} &  (X_n^d)' \arrow[swap]{r}{(j_{n+1}^n)'}         
      & (X_{n+1}^d)' \arrow{r} &\cdots
    \end{tikzcd}
  \end{equation}
  by Lemma \ref{PROP-2-Banach}. 
  Indeed, for $x\in X$ we have 
  \begin{align*}
    \Psi_{n+1}((i_n^{n+1})^{\star}((i_{n})^{\star}(x))) &= \Psi_{n+1}((i_{n+1})^{\star}(x))  = (j_{n+1})'(x)\\
                                                        &=(j_{n+1}^{n})'((j_n)'(x)) = (j_{n+1}^{n})'(\Psi_n((i_n)^{\star}(x)))
  \end{align*}
  because $(i_{n+1})^{\star}=(i_{n}^{n+1})^{\star}\circ (i_n)^{\star}$ and $(j_{n+1})'=(j_{n+1}^{n})'\circ (j_n)'$. Since
  $(i_n)^{\star}(X)$ is a dense subspace of $(X_n^d)^{\star}$ the commutativity follows.

  \medskip

  Using these facts, we are now able to to prove the claimed duality result.
  By~\cite[25.12, 25.14 and 25.15]{MV}, $\proj{n\in\mathbb{N}} X_n^d$ is a reflexive space, being a projective limit of reflexive Banach spaces, and therefore it is distinguished; in particular its dual is isomorphic to $\ind{n\in\mathbb{N}} (X_n^d)'$ by~\cite[Definition~9 and Proposition~10, pp.~84--85]{BIntroduction}.
\end{proof}

\begin{rem}
  We mention that \eqref{PROP-3} yields $\ind{n\in\mathbb{N}}X_{-n}\cong \ind{n\in\mathbb{N}}(X_n^d)^{\star}$, even if the limit spaces are not Hausdorff. In particular the universal extrapolation space $X_{-\infty}=(\ind{n\in\mathbb{N}}X_{-n})^{\hatt}$ in the sense of \cite{W} exists if the sequence of the pivot duals is known to have a Hausdorff inductive limit.
\end{rem}

\medskip

In the case where $X$ is not only a Banach space but a Hilbert space and the operator $A$ is in addition self-adjoint, we may get the following stronger result. For a Hilbert space $X$, in order to avoid confusion with the duality mapping, we use the notation $(\cdot,\cdot)$ for the scalar product of $X$.

\medskip

\begin{thm}\label{THM-Hilbert} 
  Let $X$ be a Hilbert space and $A\colon D(A)\rightarrow X$ be a densely defined, self-adjoint operator such that $A^{-1}$ exists and belongs to $L(X)$. Then we have
  \begin{equation}
    (\proj{n\in\mathbb{N}}X_n)'_b\cong \ind{n\in\mathbb{N}} X_{-n}.
  \end{equation}
  In particular, the inductive limit $\ind{n\in\mathbb{N}} X_{-n}$ is complete.
\end{thm}

\begin{proof}
  In order to prove this theorem, we want to show that for all $n\in\mathbb{N}$ the spaces $X_n$ and $X_n^d$ are isomorphic in a way which is compatible with the construction in Section~\ref{SEC-2} and the proof of Theorem~\ref{THM}. In other words, we have to find an isomorphism between $X$ and $X'$ which is ``canonical enough for our purposes''.

  \medskip

  Since $A$ is self-adjoint, by the spectral theorem (see eg.~\cite[Theorem~D.5.1]{Haase-Buch}) there is an measure space $(\Omega,\mu)$, a real valued continuous function $f$ on $\Omega$ and a unitary operator $U\colon X\to L^{2}(\Omega,\mu)$ mapping $D(A)$ onto $D(M_f)$ and satisfying 
  \begin{equation}\label{eq:SpThm}
    A=U^{\star}\circ M_f \circ U, 
  \end{equation}
  where $M_fg:= fg$ for $g\in D(M_f)$.

  \medskip

  The conjugation mapping $L^{2}(\Omega,\mu)\to L^{2}(\Omega,\mu), f\mapsto \bar{f}$, where $\bar{f}(t):=\overline{f(t)}$, is an antilinear isometry and hence the mapping
  \begin{equation}\label{eq:Conj}
    J\colon X \to X,\; x\mapsto U^{\star}(\overline{U(x)}) 
  \end{equation}
  is an antilinear isometry onto. As the domain $D(M_f)$ is invariant under complex conjugation, i.e., $\overline{D(M_f)}=D(M_f)$, and since $U^{\star}(D(M_f))=D(A)$, we may deduce $J(D(A)) = D(A)$ and inductively $J(X_n)=X_n$ for all $n\in\mathbb{N}$. 
  The identites
  \begin{equation}
    \begin{aligned}
      (J\circ A)(x) & = U^{\star} (\overline{UAx}) = U^{\star}(\overline{(U\circ U^{\star}\circ M_f\circ  U)(x)}) 
      = U^{\star}(\overline{fU(x)}) = U^{\star}(f\overline{U(x)})\\
      &= (U^{\star}\circ M_{f}\circ U)(U^{\star}(\overline{U(x)}))= (A\circ J)(x),
    \end{aligned}
  \end{equation}
  which follow from~\eqref{eq:SpThm}, \eqref{eq:Conj},~and the fact that $f$ is real valued, show that $A$ and $J$ commute.

  \medskip

  We denote by $J_X\colon X \to X', x\mapsto (\cdot, x)$ the canonical map, which, by the Fr\'{e}chet-Riesz Theorem, is an antilinear, isometric map onto. Hence the composition
  \begin{equation}
    J_X\circ J \colon X \to X', \quad x \mapsto (\cdot, J(x))
  \end{equation}
  is a linear isometry onto. 

  \medskip

  Note that a direct computation shows that $J_X$ maps $D(A)=D(A^\star)$ onto $D(A')$ and hence inductively $X_n$ onto $X_n^d$.
  In addition, it is easy to see that $A=J_X^{-1}\circ A'\circ J_X$. From this we may conclude that $J_X\circ J$ maps $X_n$ onto $X_n^d$.
  Moreover the computation
  \begin{equation}
    \begin{aligned}
      \|(J\circ J_{X}^{-1})(\varphi)\|_{n} & = \|J(J_{X}^{-1})(\varphi)\|_{X} + \|A^{n}(J\circ J_{X}^{-1})(\varphi)\|_{X} 
      = \|J_{X}^{-1}(\varphi)\|_{X} + \|A^{n}(J_{X}^{-1}(\varphi))\|_{X} \\
      & = \|\varphi\|_{X'} + \|(J_{X}\circ A\circ J_{X}^{-1})\circ \cdots\circ (J_{X}\circ A\circ J_{X}^{-1})\varphi\|_{X'} \\
      & = \|\varphi\|_{X'} + \|(A')^{n}\varphi\|_{X'} = \|\varphi\|_{X_n^{d}}
    \end{aligned}
  \end{equation}
  for $\varphi\in X_{n}^{d}$, shows that $J_{X}\circ J|_{X_n}$ is even an isometry onto. Since we are only using restrictions of 
  the mapping $J_{X}\circ J$, the diagram 
  \begin{equation}
    \begin{tikzcd}
      \cdots \arrow{r} & X_{n+1} \arrow[swap]{d}{J_X\circ J|_{X_{n+1}}}\arrow{r} & X_{n} \arrow[swap]{d}{J_X\circ J|_{X_n}} \arrow{r} 
      & \cdots \arrow{r} & X \arrow[swap]{d}{J_X\circ J}\\
      \cdots \arrow{r} & X_{n+1}^{d}\arrow{r} & X_{n}^{d} \arrow{r} & \cdots \arrow{r} & X',
    \end{tikzcd}
  \end{equation}
  where the horizontal arrow are inclusions, is commutative. From this we may deduce that 
  \[
    \proj{n\in\mathbb{N}}{X_n} \cong \proj{n\in\mathbb{N}}{X_n^d},
  \]
  which in combination with Theorem~\ref{THM} finishes the proof.
\end{proof}

\medskip

We conclude this section with two corollaries. The first is the existence result for the universal ex\-tra\-po\-lation space, cf.~the last paragraph of Section \ref{SEC-1}.

\begin{cor} Let $X$ be a reflexive Banach space and $A\colon D(A)\rightarrow X$ be the generator of a \Cnull-semigroup such that $A^{-1}$ exists and belongs to $L(X)$. Then the inductive limit $\ind{n\in\mathbb{N}} X_{-n}$ is complete and $X_{-\infty}=\ind{n\in\mathbb{N}} X_{-n}$ holds.\hfill\qed
\end{cor}

The second corollary fills the gap that we left in our initial example of the classical Sobolev scale considered in Section \ref{SEC-1}.

\begin{cor}\label{COR-2} 
  Let $X$ be a Hilbert space and $A\colon D(A)\rightarrow X$ be the generator of a \Cnull-semigroup such that $A^{-1}$ 
  exists and belongs to $L(X)$. If there is a $k\in\mathbb{N}$ such that $A^k\colon D(A^k)\to X$ is self-adjoint, then 
  \[
    (\proj{n\in\mathbb{N}}X_n)'_b\cong \ind{n\in\mathbb{N}} X_{-n}
  \]
  holds.
\end{cor}

\begin{proof}
  \red{Since $A^k$ is self-adjoint, we can apply Theorem~\ref{THM-Hilbert} to $A^k$ by Lemma~\ref{LEM-0}(ii). The 
  Corollary follows since the inductive sequences $(X_n,i_{n}^{n+1})$ and $(X_{kn}, i_{kn+k-1}^{k(n+1)}\circ\cdots\circ i_{nk}^{nk+1})$ 
  are equivalent and hence the projective limits $\proj{n\in\mathbb{N}}{X_n}$ and $\proj{n\in\mathbb{N}}{X_{kn}}$ coincide.}
\end{proof}

\smallskip


\section{Examples and open problems}

\smallskip

With the above theory we are able to extend our initial Hilbert space example form Section \ref{SEC-1} to arbitrary Lebesgue exponent $p\in(1,\infty)$.

\smallskip

We consider $X=L^{p}(\mathbb{R})$ for $1<p<\infty$ and $A=\frac{\rm d}{{\rm d}x}$ with $D(A)=\bigl\{f\in L^p(\mathbb{R})\:;\:\frac{\rm d}{{\rm d}x}f\in L^p(\mathbb{R})\bigr\}$. We get $X_n=W^{p,n}(\mathbb{R})$ and $X_{\infty}=\mathcal{D}_{L^p}(\mathbb{R})$. As duals with respect to the pivot space $X=L^{p}(\mathbb{R})$, we get $X_n^d=W^{q,n}(\mathbb{R})$, where $1/p+1/q=1$ and $(X_n^d)^\star=W^{p,-n}(\mathbb{R})$. Indeed, the completion of $L^{p}(\mathbb{R})$ with respect to the norm $\|\cdot\|_{\star,n}$ is a classical characterisation of the space $W^{p,-n}(\mathbb{R})$, cf.~Adams, Fournier \cite[3.13 on p.~64]{Adams}. From Lemma \ref{LEM-1} we now get that $X_{-n}=W^{p,-n}(\mathbb{R})$. Finally, Theorem~\ref{THM} implies that
\[
  X_{-\infty}=\ind{n\in\mathbb{N}}W^{p,-n}(\mathbb{R}) = \mathcal{D}'_{L^p}(\mathbb{R})
\]
since $\proj{n\in\mathbb{N}} (X_n^d)' = \proj{n\in\mathbb{N}} W^{q,n}(\mathbb{R}) = \mathcal{D}_{L^q}(\mathbb{R})$ and $\mathcal{D}'_{L^p}(\mathbb{R})=(\mathcal{D}_{L^q}(\mathbb{R}))'$. The case $p=2$ is the classical Sobolev scale considered in the introduction.

\medskip

Unfortunately, the question of how to deal with the case of non-reflexive Banach spaces remains open and therefore other typical spaces on which for instance the shift semigroup can be studied cannot be treated yet. However, also here the universal interpolation space can appear to be a well-studied space. For example in the case $X=\mathcal{C}_{0}(\mathbb{R})$ and $A=\frac{\rm d}{{\rm d}x}$ with $D(A)=\mathcal{C}_0^1(\mathbb{R})$, we get $X_n = \mathcal{C}^{n}_0(\mathbb{R})$ and $X_{\infty}=\mathcal{C}^{\infty}_{0}(\mathbb{R})$ which was again studied by Schwartz \cite{Schwartz} under the name $\dot{\mathcal{B}}(\mathbb{R})$ as a predual of $\mathcal{D}'_{L^{1}}(\mathbb{R})$. The well known space of distributions vanishing at infinity is the space
\[
\dot{\mathcal{B}}'(\mathbb{R}) =\overline{\mathcal{E}'(\mathbb{R})}^{\mathcal{D}'_{L^\infty}},
\]
whose dual is $\mathcal{D}_{L^{1}}(\mathbb{R})$, see e.g.~\cite[Proposition~7]{Bar2015Completing}. Note that $\dot{\mathcal{B}}'(\mathbb{R})$ is not the dual space of $\dot{\mathcal{B}}(\mathbb{R})$. It seems to be natural to conjecture that $\dot{\mathcal{B}}'(\mathbb{R})$ coincides with the universal extrapolation space $X_{-\infty}$.

\bigskip

\footnotesize

{\sc Acknowledgements. }The authors wish to thank Markus Haase for constructive comments and discussions that improved some of the results and the readability of the whole paper. In addition, the authors would like to thank an anonymous referee for his/her valuable comments.

\normalsize

\bigskip

\normalsize

\bibliographystyle{amsplain}

\bibliography{BW}

\end{document}